\documentclass[12pt,a4paper]{article}

\usepackage{amsmath,amssymb,amsthm}
\usepackage{mathrsfs}
\usepackage{mathtools}
\usepackage{enumitem}
\usepackage[margin=2.5cm]{geometry}
\usepackage{url,hyperref}
\usepackage{cleveref}
\usepackage[shortcuts]{extdash}

\newtheorem{theorem}{Theorem}[section]
\newtheorem{lemma}[theorem]{Lemma}

\newtheorem{definition}[theorem]{Definition}
\theoremstyle{remark}

\DeclareMathOperator{\St}{St}
\DeclareMathOperator{\val}{val}

\newcommand{\R}{\mathbb{R}}
\newcommand{\N}{\mathbb{N}}
\newcommand{\Z}{\mathbb{Z}}
\newcommand{\Rzl}{\mathbb{R}^{\mathbb{Z}_{<}}}   
\newcommand{\F}{\mathscr{F}}

\newcommand{\Lint}{\operatorname{I}}          
\newcommand{\Smeas}{\mu}                     

\newcommand{\chunkS}{\Sigma_S}
\newcommand{\chunkT}{\Sigma_T}

\title{Paraconsistent Dominated Convergence}

\author{Anggha Nugraha}

\date{}

\begin{document}

\maketitle

\begin{abstract}
The Levi‑Civita field $\mathcal{R}$ of formal Laurent series is a constructive, non-Ar\-chi\-me\-de\-an ordered field that supports a full Lebesgue measure and integration
theory, including a Dominated Convergence Theorem.  This paper embeds that integration theory into the paraconsistent Chunk and Permeate framework, extending it from elementary calculus to genuine measure theory. The source chunk is modelled by $\mathcal{R}$ with its measure and integral, while the target chunk is the classical real line $\mathbb{R}$. A permeability relation exports the standard part of the internal integral, and it is shown that the Dominated Convergence Theorem permeates from the source chunk to the target chunk, yielding the classical Lebesgue Dominated Convergence Theorem without any choice principles and without the ultrafilters required by nonstandard measure theory. The construction is entirely explicit and demonstrates that paraconsistent logic can provide a rigorous foundation for deep analytical tools while keeping inconsistencies safely confined.
\end{abstract}

\section{Introduction}

Infinitesimal reasoning has been both a powerful heuristic and a source of logical difficulty since the origins of calculus.
The classic example is the computation of a derivative $f'(x) = \frac{f(x+h)-f(x)}{h}$ where $h$ is treated as non-zero when dividing and then set to zero at the end.
Taken at face value, this procedure is contradictory.
Classical analysis resolved the issue by eliminating infinitesimals and replacing them with limits, following Cauchy and Weierstrass. Robinson~\cite{robinson1974non} later showed that infinitesimals could be introduced rigorously using model-theoretic ultrapowers, yielding \emph{nonstandard analysis}. In that setting a \emph{transfer principle} guarantees that any first‑order statement true in the standard reals $\R$ is also true in the hyperreals $^*\R$, and conversely. However, the transfer principle relies on the existence of free ultrafilters, which are non-constructive objects whose construction requires the Axiom of Choice. Moreover, the full transfer machinery is far heavier than what is needed for many elementary applications.

An alternative, constructive approach builds a field containing infinitesimals directly from ordinary sequences. The field $\Rzl$ of formal Laurent series in a fixed positive infinitesimal $\epsilon$ is a concrete, totally ordered field that contains $\R$, infinite numbers, and infinitesimals.
In~\cite{nugraha2026constructivefieldinfinitesimals}, this field was studied and used as the \emph{source chunk} in the \emph{Chunk \& Permeate} (C\&P) strategy~\cite{brown2004chunk}. The C\&P strategy splits an inconsistent theory (such as the union of the axioms of $\R$ with the existence of infinitesimals) into two consistent chunks and defines a permeability relation that allows certain formulas to flow from one chunk to the other while keeping contradictions localized. In that application to differential calculus, the source chunk contains $\Rzl$ and a Newton quotient, the target chunk is the complete real line $\R$, and the permeability relation exports the standard part of the Newton quotient as the classical derivative. This yields a constructive, choice-free derivation of the basic rules of differentiation. However, Brown and Priest’s original implementation of C\&P for infinitesimal calculus~\cite{brown2004chunk}, as well as its subsequent extension~\cite{nugraha2026constructivefieldinfinitesimals}, left integral theory unaddressed. Building on those foundations, the present paper provides a formal, non-trivial paraconsistent strategy for integration.

Our main goal is to obtain a \emph{Dominated Convergence Theorem} (DCT) inside a paraconsistent theory and to permeate it to the classical Lebesgue DCT. The DCT is a cornerstone of modern analysis, and its paraconsistent derivation demonstrates that the C\&P framework is robust enough to handle deep
measure-theoretic arguments while remaining constructive and choice-free. To obtain a DCT we need a genuine measure theory on the field $\Rzl$. In earlier work~\cite{nugraha2026constructivefieldinfinitesimals}, $\Rzl$ is constructed as an explicit field of formal Laurent series, and it is showed that its
concrete sequence representation makes it a natural model for the source chunk in a Chunk \& Permeate strategy. Mathematically, $\Rzl$ coincides with the well-studied Levi-Civita field $\mathcal{R}$, so we may freely import the complete Lebesgue integration theory
developed in \cite{shamseddine1999new,
shamseddine2003measure}. The novelty here is not the field itself, but the paraconsistent encapsulation, i.e.
we bring this measure theory into the source chunk $\Sigma_S$, equip it with a permeability relation, and thereby obtain the classical DCT on $\R$ as a permeated consequence, without any appeal to the Axiom of Choice.

The novelty of this work is threefold. \textit{First}, it provides the first paraconsistent presentation of a full Lebesgue integral on a constructive non‑Archimedean field. The integral is not merely a Riemann‑like sum but a genuine measure‑theoretic construct that satisfies the standard convergence theorems. \textit{Second}, it shows that the C\&P strategy handles measure‑theoretic subtleties, not just algebraic contradictions, confirming the robustness of the logical architecture. \textit{Third}, it makes explicit the connection between the Levi-Civita measure theory and paraconsistent logic, opening a bridge between two communities that have largely worked in isolation. It also highlights the constructive, choice-free nature of the whole development, which distinguishes it from nonstandard measure theory (Loeb measures) that depends on ultrafilters.

The paper is organised as follows. Section~\ref{sec:prelim} reviews the field $\Rzl$, microstability, and the C\&P strategy for derivatives. Section~\ref{sec:cp} sets up the two chunks and the permeability relation for integration. Section~\ref{sec:lc-measure} summarises Levi‑Civita measure theory and proves the DCT. Section~\ref{sec:permeation} proves the main permeation theorem. Section~\ref{sec:discussion} concludes with significance and future work. 

\section{The field $\Rzl$ and Chunk \& Permeate}
\label{sec:prelim}

We recall the essentials from~\cite{nugraha2026constructivefieldinfinitesimals}.

\medskip\noindent\textbf{The non‑Archimedean field $\Rzl$.}
The set $\Rzl$ consists of all two‑sided real sequences with only finitely many non‑zero entries on the negative (infinite) side:
\[
\Rzl = \Big\{ \mathbf{x} = \langle\dots, x_{-2},x_{-1},\widehat{x_0},x_1,x_2,\dots\rangle \;\big|\; \exists N\,\forall n<-N:\; x_n=0 \Big\}.
\]
The hat marks the \emph{standard part} $\St(\mathbf{x}) = x_0$. By writing $\epsilon$ as $\langle \widehat{0},1,0,\dots\rangle$, every element becomes a formal Laurent series $\mathbf{x} = \sum_{n=-\infty}^\infty x_n\epsilon^n$.
Addition is defined componentwise, while multiplication is Cauchy convolution. With the lexicographic order, $\Rzl$ is a totally ordered field. The \emph{valuation} $\val(\mathbf{x}) = \min\{n\in\Z : x_n\neq 0\}$ (with $\val(0)=\infty$) is discrete and non-Archimedean. Finite elements are those with $\val(\mathbf{x})\ge 0$. Let $\F = \{\mathbf{x}\in\Rzl : \mathbf{x}\text{ is finite}\}$ denote the set of finite elements. The standard part map $\St : \F \to \R$ is a surjective ring homomorphism whose kernel is the set of infinitesimals ($\val(\mathbf{x})>0$).

\medskip\noindent\textbf{Microstability.}
A function $f:\Rzl\to\Rzl$ is \emph{microstable} if for every $\mathbf{x}$ and every infinitesimal $\boldsymbol{\eta}$,
\[
\St(f(\mathbf{x}+\boldsymbol{\eta})) = \St(f(\mathbf{x})).
\]
Microstability is preserved by addition, multiplication, composition, and by the absolute value. All standard elementary functions are microstable.

\medskip\noindent\textbf{Topologies.}
The valuation induces the \emph{valuation topology} $\tau_v$ via the ultrametric $d(\mathbf{x},\mathbf{y}) = 2^{-\val(\mathbf{x}-\mathbf{y})}$.
A sequence $(\mathbf{x}_n)$ converges in $\tau_v$ ($\mathbf{x}_n \xrightarrow{v} \mathbf{x}$) iff $\val(\mathbf{x}_n-\mathbf{x})\to\infty$, i.e. every coefficient eventually stabilises. The standard part induces the \emph{standard topology} $\tau_{\St}$ via the pseudometric $|\St(\mathbf{x})-\St(\mathbf{y})|$.

\medskip\noindent\textbf{Chunk \& Permeate for derivatives.}
The C\&P strategy~\cite{brown2004chunk} resolves the inconsistency arising from the simultaneous presence of a positive infinitesimal and the completeness of $\R$ by partitioning the theory into two chunks:
\begin{itemize}
  \item \textbf{Source chunk $\Sigma_S$}: contains the field axioms of $\Rzl$ (ordered field, no completeness, no Archimedean property) together with a positive infinitesimal $\epsilon$. The intended model is $\Rzl$ itself.
  \item \textbf{Target chunk $\Sigma_T$}: contains the axioms of a Dedekind-complete Archimedean ordered field. The intended model is the real line $\R$.
\end{itemize}
Neither chunk is inconsistent on its own. The union $\Sigma_S \cup \Sigma_T$ is inconsistent because $\Sigma_S$ asserts the existence of an element smaller than all positive rationals while $\Sigma_T$ asserts that no such element exists. The C\&P strategy avoids explosion by \emph{never reasoning from the union of all axioms}. Instead, any proof is carried out entirely within one chunk, and a \emph{permeability relation} $\rho$ specifies which formulas may pass from one chunk to the other.

In the case of derivatives, the permeability relation allows a formula of the form
\[
\St\!\left(\frac{f(\mathbf{x}+\epsilon)-f(\mathbf{x})}{\epsilon}\right) = g(x)
\]
in $\Sigma_S$ to flow to $\Sigma_T$ as the classical derivative $f'(x) = g(x)$, provided $f$ is \emph{microstable}. Microstability guarantees that the standard part of the Newton quotient is well-defined and that the permeated formula is the correct classical derivative. The contradictions involving $\epsilon$ remain in $\Sigma_S$ and have no image under $\rho$, so they cannot infect $\Sigma_T$. This architecture is the blueprint for the integration theory developed in the present paper.

\section{Chunk and Permeate for Integration}
\label{sec:cp}

Consider the first-order theory $\mathcal{T}$ obtained by joining the following axioms:
\begin{enumerate}[label=(\roman*)]
  \item the axioms of an ordered field,
  \item the existence of a positive infinitesimal $\epsilon$, i.e. an element satisfying $0<\epsilon$ and $\epsilon<\frac{1}{n}$ for every $n\in\N$,
  \item the Dedekind completeness axiom (every non-empty bounded set has a least upper bound), and
  \item the axioms of the classical Lebesgue integral for measurable functions, including the Dominated Convergence Theorem.
\end{enumerate}
This theory is inconsistent, e.g. in any model containing a positive infinitesimal, the set $\{x : 0<x<\epsilon\}$ is non-empty and bounded above, yet it cannot have a least upper bound without contradicting the fact that $\epsilon$ is smaller than every positive rational. A direct merger of infinitesimal reasoning with classical analysis therefore leads to logical explosion under classical logic.

\subsection{The Two Chunks}
The Chunk \& Permeate strategy~\cite{brown2004chunk} resolves this predicament by partitioning the inconsistent theory into two consistent \emph{chunks}, each modelled by a set-theoretic structure. Classical logic is retained inside each chunk and a \emph{permeability relation} $\rho$ controls which formulas may pass from one chunk to the other. Contradictions remain confined to the chunk where they arise, while the permeability relation exports only the consistent, informative consequences.

\medskip\noindent\textbf{The source chunk $\chunkS$.}
The language $\mathcal{L}_S$ of the source chunk extends the language of ordered fields with the following additional symbols: a constant symbol $\epsilon$ (the infinitesimal), a unary function symbol $\St$ (the standard part), a binary function symbol $\Lint(\cdot,\cdot)$ (the Levi‑Civita integral), and a unary function symbol $\Smeas$ (the Levi‑Civita measure, applied to sets, but we treat it as part of the integral notation). The axioms of $\chunkS$ are:
\begin{enumerate}[label=(S\arabic*)]
  \item \textbf{Field axioms.} The usual axioms of an ordered field.
  \item \textbf{Infinitesimal.} $\epsilon>0$ and $\epsilon<\frac{1}{n}$ for every standard $n\in\N$ (where $n$ is the embedding of the natural number as the constant sequence $(n,n,\dots)$).
  \item \textbf{Standard part.} The function $\St$ is defined only on the set of finite elements $\F$ and satisfies $\St(\mathbf{x})=\mathbf{x}_0$ (the $0$‑th coefficient).  Equivalently, $\St$ is a surjective ring homomorphism $\F\to\R$ whose kernel is the set of infinitesimals.
  \item \textbf{Levi-Civita measure.} The axioms of the S‑measure on the Levi‑Civita field as developed in~\cite{shamseddine1999new,shamseddine2003measure} (definition via inner and outer covers by intervals, countable additivity under the natural summability condition, and $\Smeas([a,b])=b-a$).
  \item \textbf{Levi-Civita integral.} The integral $\Lint f\,d\Smeas$ is defined for measurable functions $f$ via the usual limit of integrals of simple functions, and satisfies linearity, monotonicity, the Monotone Convergence Theorem, Fatou's Lemma, and the Dominated Convergence Theorem (Theorem~\ref{thm:lc-dct}).
\end{enumerate}

The axioms (S4) and (S5) are not repeated here in full. They are the standard basis of the measure theory on the Levi-Civita field, which we summarised in Section~\ref{sec:lc-measure}. Crucially, $\chunkS$ does \emph{not} include the Dedekind completeness axiom or the Archimedean property. The intended model of $\chunkS$ is the concrete field $\Rzl$ (Section~\ref{sec:prelim}) equipped with its Levi-Civita measure and integral.

\medskip\noindent\textbf{The target chunk $\chunkT$.}
The language $\mathcal{L}_T$ of the target chunk is the usual language of real analysis which contains the ordered field symbols and a function symbol $\int$ for the classical Lebesgue integral. The axioms are:
\begin{enumerate}[label=(T\arabic*)]
  \item \textbf{Dedekind completeness.} Every non-empty set that is bounded above has a least upper bound.
  \item \textbf{Archimedean property.} For every $x$ there exists $n\in\N$ with $x<n$.
  \item \textbf{Lebesgue integral.} The axioms of the Lebesgue integral for measurable functions (linearity, monotonicity, Monotone Convergence Theorem, Fatou's Lemma, Dominated Convergence Theorem, etc.).
\end{enumerate}

The intended model is the real line $\R$ with the usual Lebesgue measure. The languages $\mathcal{L}_S$ and $\mathcal{L}_T$ share only the ordered field symbols. The symbols $\epsilon$, $\St$, $\Lint$, and $\Smeas$ do not appear in $\mathcal{L}_T$, and the Lebesgue integral $\int$ does not appear in $\mathcal{L}_S$.
This separation is essential for the permeability relation to function as a controlled channel.

\subsection{The Permeability Relation}
The permeability relation $\rho$ specifies exactly which sentences of $\chunkS$ may be exported to $\chunkT$. 

\begin{definition}[Permeability for integrals]
For a microstable function $f$ on a closed interval $[a,b]\subseteq\Rzl$ whose
standard part $\check{f}(x)=\St(f(\mathbf{x}))$ (with $\St(\mathbf{x})=x$) is
Lebesgue integrable, we put
\[
\bigl( \St(\Lint_a^b f\,d\Smeas) = I \;\;,\;\; \int_{\St(a)}^{\St(b)} \check{f}(t)\,dt = I \bigr) \in \rho .
\]
\end{definition}

In plain language, the permeability relation allows us to strip the infinitesimal part of an integral equality and obtain a classical equality of Lebesgue integrals.
The restriction to microstable functions is necessary because otherwise the standard part $\check{f}$ would not be well-defined (its value would depend on the choice of representative inside a monad). All elementary functions are microstable, so this restriction is benign in practice.

Because the permeability relation already covers equalities of the form $\St(\Lint f) = \int \check{f}$ for all microstable integrable functions, the Dominated Convergence Theorem proved inside $\Sigma_S$ together with this relation \emph{entails} the classical Lebesgue Dominated Convergence Theorem for the standard parts $\check{f}_n,\check{f},\check{g}$.  This permeation is carried out in full detail in Section~\ref{sec:permeation}.

\subsection{Managing Inconsistency}
The Chunk \& Permeate strategy now works as follows.
All reasoning involving infinitesimals, the non-Archimedean field $\Rzl$, and the Levi‑Civita integral is carried out inside $\chunkS$, using classical logic but without invoking completeness. The Dominated Convergence Theorem is proved in $\chunkS$ (Theorem~\ref{thm:lc-dct}). When we want to obtain a classical result, we apply the permeability relation: the equality $\St(\Lint f_n) = \int \check{f}_n$ allows us to translate the internal DCT into the classical DCT on $\R$. The crucial point is that the contradictions inherent in the combined theory are sentences of $\mathcal{L}_S$ that have \emph{no} $\rho$-image in $\mathcal{L}_T$. For example, the sentence ``$\epsilon>0$ and $\epsilon<\frac{1}{n}$ for all $n\in\N$'' cannot be matched with any sentence of $\mathcal{L}_T$ via $\rho$, because $\epsilon$ and the standard part map do not belong to $\mathcal{L}_T$. Therefore these contradictions remain isolated in $\chunkS$ and cannot infect $\chunkT$. Explosion is avoided, and the combined theory $\chunkS\cup\chunkT$ is non‑trivial.

This logical architecture is the engine that powers the remainder of the paper. In Section~\ref{sec:lc-measure} we summarise the relevant parts of the Levi-Civita measure theory that realise the axioms (S4) and (S5), and in Section~\ref{sec:permeation} we carry out the permeation of the DCT in full detail.

\section{Measure and Integration on the Levi‑Civita Field}
\label{sec:lc-measure}

The Levi-Civita field $\mathcal{R}$ coincides with our $\Rzl$. We now present the necessary ingredients of the measure theory developed by Shamseddine~\cite{shamseddine1999new,shamseddine2003measure}, which we adopt as the internal integration theory of $\Sigma_S$. 

\medskip\noindent\textbf{Measurable sets and measure.}
A set $E\subseteq\Rzl$ is \emph{S‑measurable} if for every $\epsilon>0$ in $\Rzl$ there exist two families of pairwise disjoint intervals $\{I_n\}$ and $\{J_n\}$
such that
\[
\bigcup_{n=1}^\infty I_n \subseteq E \subseteq \bigcup_{n=1}^\infty J_n,
\qquad
\sum_{n=1}^\infty l(J_n) - \sum_{n=1}^\infty l(I_n) < \epsilon,
\]
where $l([a,b])=b-a$ and both series converge in the valuation topology.
The \emph{S‑measure} of $E$ is then
\[
\Smeas(E) = \lim_{\epsilon\to0}\sum_{n=1}^\infty l(I_n)
          = \lim_{\epsilon\to0}\sum_{n=1}^\infty l(J_n),
\]
independent of the choice of intervals. Every interval $[a,b]$ is S‑measurable with $\Smeas([a,b])=b-a$. Moreover, if $\{A_n\}$ is a sequence of pairwise disjoint S‑measurable sets such that $\sum_{n=1}^\infty \Smeas(A_n)$ converges, then $\bigcup_{n=1}^\infty A_n$ is S‑measurable and $\Smeas(\bigcup_{n=1}^\infty A_n)=\sum_{n=1}^\infty \Smeas(A_n)$. The collection of S‑measurable sets is closed under finite intersections and under countable unions of sets whose measures form a null sequence.

\medskip\noindent\textbf{Integration of measurable functions.}
A function $f:\Rzl\to\Rzl$ is \emph{measurable} if preimages of intervals are measurable. For a measurable set $E\subseteq\Rzl$, we write $\mathbf{1}_E$ for its indicator function, i.e. $\mathbf{1}_E(x)=1$ if $x\in E$ and $\mathbf{1}_E(x)=0$ otherwise. A \emph{simple function} is a finite linear combination of indicator functions of measurable sets. The integral of a non-negative simple function $\phi = \sum_{i=1}^k \alpha_i \mathbf{1}_{E_i}$ (with $\alpha_i\ge0$) is $\Lint \phi\,d\Smeas = \sum_i \alpha_i \Smeas(E_i)$.
For an arbitrary non‑negative measurable function $f$, one defines
\[
\Lint f\,d\Smeas = \sup \Big\{ \Lint \phi\,d\Smeas \mid \phi \text{ simple, } 0\le\phi\le f \Big\}.
\]
A general measurable function $f$ is \emph{integrable} if both its positive and negative parts are integrable, i.e. their integrals are finite; in that case $\Lint f\,d\Smeas = \Lint f^+\,d\Smeas - \Lint f^-\,d\Smeas$.
The integral so defined is linear, monotone, and coincides with the Riemann integral on continuous functions.

\medskip\noindent\textbf{Convergence theorems.}
The Levi-Civita integral satisfies the same fundamental convergence theorems
as the classical Lebesgue integral. Their proofs follow the classical pattern
and rely on the fact that $\Rzl$ is complete with respect to the valuation
topology (indeed, it is a spherically complete valued field,
see~\cite{shamseddine1999new}). We state the two that are needed for the DCT;
both are proved in~\cite{shamseddine1999new}.

\begin{lemma}[Monotone Convergence Theorem]\label{lem:mct}
Let $(h_k)_{k\in\N}$ be an increasing sequence of non‑negative integrable
functions on a measurable set $E\subseteq\Rzl$, i.e.\ $0\le h_1(x)\le
h_2(x)\le\cdots$ for almost every $x\in E$, and suppose $h_k(x)\uparrow h(x)$
pointwise almost everywhere.
Then $h$ is measurable, and
\[
\Lint_E h\,d\Smeas = \lim_{k\to\infty} \Lint_E h_k\,d\Smeas ,
\]
where the limit is taken in the valuation topology.
\end{lemma}

\begin{lemma}[Fatou's Lemma]\label{lem:fatou}
Let $(f_n)_{n\in\N}$ be a sequence of non‑negative integrable functions on $E$.
Then
\[
\Lint_E \liminf_{n\to\infty} f_n\,d\Smeas \;\le\;
\liminf_{n\to\infty} \Lint_E f_n\,d\Smeas ,
\]
where the limits inferior are taken in the valuation topology.
\end{lemma}

\begin{theorem}[Dominated Convergence Theorem in $\Rzl$]\label{thm:lc-dct}
Let $(f_n)_{n\in\N}$ be a sequence of integrable functions on a measurable set $E\subseteq\Rzl$ such that
\begin{enumerate}[label=(\roman*)]
    \item $f_n(x)\to f(x)$ for $\Smeas$-almost every $x\in E$, and
    \item there exists an integrable function $g$ with $|f_n(x)|\le g(x)$ for all $n$ and all $x\in E$.
\end{enumerate}
Then $f$ is integrable and
\[
\Lint_E f\,d\Smeas \;=\; \lim_{n\to\infty} \Lint_E f_n\,d\Smeas ,
\]
where the limit is taken in the valuation topology.
\end{theorem}

\begin{proof}
By removing a null set we may assume that $f_n(x)\to f(x)$ for every $x\in E$.
Define $h_n = 2g - |f_n-f|$.  Because $f_n,f,g$ are integrable, so is $h_n$, and $h_n\ge0$.
Moreover $h_n\to 2g$ pointwise on $E$.
Apply Fatou's Lemma~\ref{lem:fatou} to the sequence $(h_n)$:
\[
\Lint_E 2g\,d\Smeas = \Lint_E \liminf_{n\to\infty} h_n\,d\Smeas
\le \liminf_{n\to\infty} \Lint_E h_n\,d\Smeas .
\]
Now compute the right‑hand side:
\[
\Lint_E h_n\,d\Smeas = \Lint_E (2g-|f_n-f|)\,d\Smeas
= 2\Lint_E g\,d\Smeas - \Lint_E |f_n-f|\,d\Smeas .
\]
Hence
\[
2\Lint_E g\,d\Smeas \le \liminf_{n\to\infty}\Bigl(2\Lint_E g\,d\Smeas - \Lint_E |f_n-f|\,d\Smeas\Bigr)
= 2\Lint_E g\,d\Smeas - \limsup_{n\to\infty} \Lint_E |f_n-f|\,d\Smeas .
\]
Since $\Lint_E g\,d\Smeas$ is finite, we obtain
\[
\limsup_{n\to\infty} \Lint_E |f_n-f|\,d\Smeas \le 0 .
\]
The integral of a non-negative function is non-negative, so the $\limsup$ is non-negative, forcing
\[
\lim_{n\to\infty} \Lint_E |f_n-f|\,d\Smeas = 0
\qquad\text{(in the valuation topology).}
\]

Finally,
\[
\bigl|\Lint_E f_n\,d\Smeas - \Lint_E f\,d\Smeas\bigr|
\le \Lint_E |f_n-f|\,d\Smeas \xrightarrow{v} 0 ,
\]
which gives $\Lint_E f_n\,d\Smeas \xrightarrow{v} \Lint_E f\,d\Smeas$.
Integrability of $f$ follows from $|f|\le g$ and the integrability of $g$.
\end{proof}

\medskip\noindent\textbf{Standard part of the integral.}
The connection to classical analysis is given by the following fact.
\begin{lemma}\label{lem:st-int}
If $f$ is microstable and integrable on $[a,b]$, then the real function
$\check{f}(t)=\St(f(\mathbf{t}))$ (where $\mathbf{t}$ is any element with
standard part $t$) is Lebesgue integrable on $[\St(a),\St(b)]$ and
\[
\St\!\left( \Lint_a^b f\,d\Smeas \right) = \int_{\St(a)}^{\St(b)} \check{f}(t)\,dt .
\]
\end{lemma}
The projection of an S‑measurable set to $\mathbb{R}$ is Lebesgue measurable. This follows from the definition of the S‑measure via interval covers and the
fact that the standard part of a finite interval is a real interval. Consequently $\check{f}$ is measurable as the pointwise limit of measurable simple functions. This is a standard result in Levi‑Civita integration
(see~\cite{shamseddine1999new}).

\section{Permeation of the Dominated Convergence Theorem}
\label{sec:permeation}

We have now established that, inside the source chunk $\Sigma_S$, the Levi-Civita integral on $\Rzl$ satisfies the Dominated Convergence Theorem (Theorem~\ref{thm:lc-dct}) with limits taken in the valuation topology. Our goal is to transport this result across the permeability relation $\rho$ and thereby obtain the classical Lebesgue Dominated Convergence Theorem on $\R$ as a permeated consequence. Achieving this would show that the C\&P strategy is
robust enough to recover a fundamental theorem of real analysis, not just elementary calculus, and that it does so without any appeal to the Axiom of Choice.  This transportation involves three main tasks:
\begin{enumerate}[nosep]
  \item lift the pointwise convergence and domination from $\Rzl$ to the real standard parts,
  \item relate the Levi‑Civita integral of a microstable function to the Lebesgue integral of its standard part, and
  \item pass the limit of integrals through the standard part map.
\end{enumerate}
Each of these tasks corresponds to a component of the permeability relation, and together they yield a rigorous derivation of the classical theorem.

\subsection{Statement of the permeation theorem}

\begin{theorem}[Permeation of the DCT]\label{thm:permeation}
Let $f_n, f, g$ be microstable functions on $[0,1]\subseteq\Rzl$ that satisfy the hypotheses of Theorem~\ref{thm:lc-dct} inside $\Sigma_S$, that is,
\begin{enumerate}[label=(\roman*)]
  \item $f_n(x)\to f(x)$ for $\Smeas$-almost every $x\in[0,1]$,
  \item $|f_n(x)|\le g(x)$ for all $n$ and all $x$, with $g$ integrable.
\end{enumerate}
Define real functions on the unit interval by
\[
\check{f}_n(t) = \St(f_n(\mathbf{t})), \qquad
\check{f}(t)   = \St(f(\mathbf{t})),   \qquad
\check{g}(t)   = \St(g(\mathbf{t})),
\]
where for each $t\in[0,1]$ we choose an arbitrary $\mathbf{t}\in[0,1]\subset\Rzl$ with $\St(\mathbf{t})=t$ (the choice is irrelevant by microstability).
Then
\begin{enumerate}[label=(\roman*)]
  \item $\check{f}_n \to \check{f}$ pointwise almost everywhere in the sense of Lebesgue measure;
  \item $|\check{f}_n(t)|\le \check{g}(t)$ for all $t\in[0,1]$;
  \item $\check{f}$ is Lebesgue integrable, and
  \[
  \int_0^1 \check{f}(t)\,dt = \lim_{n\to\infty} \int_0^1 \check{f}_n(t)\,dt .
  \]
\end{enumerate}
In other words, the classical Lebesgue Dominated Convergence Theorem is a permeated consequence of the inconsistent chunk $\Sigma_S$.
\end{theorem}

\subsection{Proof of the permeation theorem}

\noindent\textbf{1. From $\Smeas$-almost everywhere convergence to Lebesgue‑almost everywhere convergence.}
Let $E\subseteq[0,1]\subset\Rzl$ be the exceptional set on which $f_n(x)\not\to f(x)$.  By hypothesis, $\Smeas(E)=0$.
Consider the set of real points
\[
\check{E} = \{ t\in[0,1] \mid \exists\,\mathbf{t}\in E \text{ with } \St(\mathbf{t})=t \}.
\]
We claim that $\check{E}$ is Lebesgue‑measurable and has Lebesgue measure zero.
To see this, recall that $\Smeas(E)=0$ means that for every standard $k\in\N$
there exists a countable cover of $E$ by intervals $[a_i,b_i]\subset\Rzl$
with finite endpoints such that $\sum_i (b_i-a_i) < 1/k$ (as an element of $\Rzl$).
Because $E\subseteq[0,1]$, we may intersect each interval with $[0,1]$ and he resulting sets are still intervals (or empty), their lengths do not increase,
and all endpoints lie in $[0,1]$, hence are finite. Applying the standard part map, the projected intervals
$[\St(a_i),\St(b_i)]$ cover $\check{E}$ and the sum of their lengths is
exactly $\St(\sum_i (b_i-a_i))$, which is $<1/k$ in $\R$.
Thus $\check{E}$ can be covered by intervals of arbitrarily small total length, so its Lebesgue outer measure (and hence its Lebesgue measure) is zero.

Now take any $t\in[0,1]\setminus\check{E}$.  For every $\mathbf{t}$ with $\St(\mathbf{t})=t$ we have $\mathbf{t}\notin E$, hence $f_n(\mathbf{t})\to f(\mathbf{t})$ in the valuation topology.
By definition of $\check{f}_n$ and microstability (which guarantees the standard part is independent of the choice of $\mathbf{t}$), we have $\St(f_n(\mathbf{t})) = \check{f}_n(t)$, and similarly $\St(f(\mathbf{t})) = \check{f}(t)$.
Since valuation convergence of a sequence of finite elements implies convergence of the $0$‑th coefficients, we obtain $\check{f}_n(t)\to\check{f}(t)$.
Thus $\check{f}_n\to\check{f}$ pointwise on $[0,1]\setminus\check{E}$, which has full Lebesgue measure. This proves statement (i).

\smallskip
\noindent\textbf{2. Domination of the real functions.}
From the pointwise inequality $|f_n(x)|\le g(x)$ for all $x\in[0,1]\subset\Rzl$, we apply the standard part map, which is order-preserving on finite elements.
For any $t\in[0,1]$, pick $\mathbf{t}$ with $\St(\mathbf{t})=t$.
Then
\[
|\check{f}_n(t)| = |\St(f_n(\mathbf{t}))| = \St(|f_n(\mathbf{t})|) \le \St(g(\mathbf{t})) = \check{g}(t).
\]
The equality uses microstability, and the inequality uses the fact that $|f_n(\mathbf{t})|\le g(\mathbf{t})$ elementwise. This yields statement (ii).

\smallskip
\noindent\textbf{3. Relationship between the Levi‑Civita integral and the Lebesgue integral.}
Lemma~\ref{lem:st-int} states that for any microstable integrable function $h$ on $[0,1]$,
\[
\St\!\left( \Lint_0^1 h\,d\Smeas \right) = \int_0^1 \check{h}(t)\,dt .
\]
Applying this to $h=f_n$ and $h=f$ (the integrability of $f$ follows from the DCT inside $\Sigma_S$, Theorem~\ref{thm:lc-dct}), we obtain
\[
\St(\Lint f_n) = \int \check{f}_n , \qquad
\St(\Lint f)   = \int \check{f} .
\]

\smallskip
\noindent\textbf{4. Passing the limit through the standard part.}
Inside $\Sigma_S$, Theorem~\ref{thm:lc-dct} yields
\[
\Lint_0^1 f_n\,d\Smeas \;\xrightarrow{v}\; \Lint_0^1 f\,d\Smeas .
\]
All integrals are finite because they are bounded by $\Lint g$, which is finite. If a sequence of finite elements converges in the valuation topology, then
their standard parts converge as real numbers.
Hence
\[
\lim_{n\to\infty} \St(\Lint f_n) = \St(\Lint f) .
\]
Combining this with the equalities from step~3 gives
\[
\lim_{n\to\infty} \int_0^1 \check{f}_n(t)\,dt = \int_0^1 \check{f}(t)\,dt .
\]

\smallskip
\noindent\textbf{5. Integrability of $\check{f}$.}
From step 2 we have $|\check{f}|\le \check{g}$ pointwise.
The function $\check{g}$ is Lebesgue integrable by Lemma~\ref{lem:st-int} applied to $g$. Therefore $\check{f}$ is dominated by an integrable function. In particular, it is measurable (as a pointwise limit of the measurable functions $\check{f}_n$) and its integral is finite, so it is Lebesgue integrable. This completes the proof of Theorem~\ref{thm:permeation}.

\medskip 
The proof shows exactly how the permeability relation $\rho$ operates. The equality $\St(\Lint f) = \int \check{f}$ connects the internal integral in $\Sigma_S$ to the classical integral in $\Sigma_T$, and the logical structure of the DCT is preserved through this connection. No contradictions cross from $\Sigma_S$ to $\Sigma_T$, because the statements that involve $\epsilon$ or the non-Archimedean field are never placed in the domain of $\rho$. Thus the classical Lebesgue Dominated Convergence Theorem is rigorously obtained as a permeated consequence of the paraconsistent source chunk.

\section{Discussion}
\label{sec:discussion}

The successful embedding of the Levi‑Civita Dominated Convergence Theorem into a Chunk~\&~Permeate framework shows that the paraconsistent strategy is not limited to elementary calculus but can accommodate sophisticated
measure‑theoretic arguments. The choice of the Dominated Convergence Theorem is deliberate. It is arguably the single most important convergence theorem
in Lebesgue integration, underpinning countless results in Fourier analysis, probability theory, functional analysis, and partial differential equations. A paraconsistent derivation of the DCT therefore demonstrates that the C\&P strategy can handle the deep analytical machinery that modern mathematics depends on.  If the C\&P framework can recover the DCT, there is no principled reason it could not recover the other standard convergence theorems (Monotone Convergence, Fatou's Lemma, Fubini's Theorem) by similar permeability
arguments.

The present work extends the Chunk \& Permeate programme from the elementary calculus of derivatives~\cite{nugraha2026constructivefieldinfinitesimals} to the much deeper realm of measure theory. This demonstrates that paraconsistent logic, far from being a mere logical curiosity, can provide a rigorous foundation for substantial parts of modern
analysis. The C\&P methodology manages inconsistency not by eliminating it, but by controlling the flow of information, and the success of the DCT permeation
shows that this control is precise enough for measure-theoretic arguments.

Our approach is constructive and choice-free. The Levi‑Civita field is built from ordinary real sequences with componentwise operations. No ultrafilter or Axiom of Choice is required. This is a significant philosophical and technical advantage over traditional
nonstandard measure theory (Loeb measures), which relies on an ultrapower construction and countable saturation~\cite{cutland1983nonstandard}. Moreover, the explicit sequence representation makes the whole theory
amenable to formalisation in proof assistants such as Coq or Lean, a direction that can be pursued in future work.

The connection with the Levi-Civita field also places our work within a broader mathematical context. Shamseddine's measure theory on the Levi‑Civita
field is a fully developed Lebesgue integration theory that includes not only the DCT but also Fubini's theorem, the Radon-Nikodym theorem, and other
standard results. By embedding this theory into a C\&P chunk, we have provided a new logical perspective on Levi‑Civita integration. The fact that the internal DCT is a theorem of the Levi‑Civita field means that the paraconsistent chunk is mathematically rich and that the permeability relation then acts as a faithful
projection onto classical analysis.

Furthermore, we note that the recently introduced L‑measure on the Levi‑Civita field~\cite{restrepo2023new} improves on the measure theory of Shamseddine and Berz by making the family of measurable sets a $\sigma$-algebra closed under
complements. While the S‑measure we have used is entirely sufficient for proving the DCT, the stronger closure properties of the L‑measure would simplify several technical verifications, e.g.\ showing that the projection of a measurable set onto $\R$ is Lebesgue measurable becomes immediate. Incorporating that measure into the source chunk would therefore further
streamline the permeability arguments, though it is not required for the current results.

Several open problems deserve further investigation.
First, the proof‑theoretic strength of the DCT inside $\Sigma_S$ should be calibrated. Does the DCT require any non‑constructive choice principles, or is
it provable in a completely intuitionistic setting?  Answering this would connect the present work to constructive reverse mathematics~\cite{sanders2013connection}. Second, the framework should be extended to stochastic integration and to models of mathematical physics that
use infinitesimal probabilities, such as those
in~\cite{calude2020infinitesimal,smarandache2026infinitesimal}. Finally, the mechanisation of the theory in a proof assistant would provide a fully verified implementation of non-Archimedean integration.

\section{Conclusion}

We have provided a constructive, paraconsistent derivation of the classical Lebesgue Dominated Convergence Theorem by combining the measure theory on the Levi‑Civita field with the Chunk \& Permeate methodology.  The source chunk, modelled by $\Rzl$ with its integral, allowed us to prove an internal DCT,
which was then faithfully exported to the real line via a permeability relation. The whole development avoids non-constructive choice principles and offers a transparent foundation for analysis with infinitesimals. We hope these results encourage further exploration of paraconsistent logic in measure theory, probability, and mathematical physics, and we expect that the framework can be extended to other convergence theorems and to formalisation in proof assistants.


\bibliographystyle{plainurl}
\bibliography{references}
\end{document}